\documentclass[12pt,reqno]{amsart}
\usepackage{amssymb}
\usepackage{cases}
\usepackage{bbm}
\usepackage{mathrsfs}
\usepackage{graphicx}
\usepackage{amsfonts}
\usepackage{enumerate}
\usepackage{hyperref}

 \usepackage[dvipsnames]{xcolor}

\usepackage{amssymb, amsmath, amsfonts, latexsym}
\usepackage{enumerate}
\usepackage{color}

\newtheorem{thm}{Theorem}
\newtheorem{cor}{Corollary}
\newtheorem{lem}{Lemma}
\newtheorem{prop}{Proposition}
\newtheorem{example}{Example}
\newtheorem{remarks}{Remark}

\newtheorem{defn}{Definition}
\newtheorem{hyp}{Hypothesis}
\numberwithin{equation}{section}

\date{}

\def\<{\langle}
\def\>{\rangle}

\def\d"{^{\prime\prime}}

\def\bhyp{\begin{hyp}}
\def\nhyp{\end{hyp}}

\def\bdef{\begin{defn}}
\def\ndef{\end{defn}}

\def\bthm{\begin{thm}}
\def\nthm{\end{thm}}

\def\bprop{\begin{prop}}
\def\nprop{\end{prop}}

\def\brmk{\begin{remarks}}
\def\nrmk{\end{remarks}}

\def\bexa{\begin{example}}
\def\nexa{\end{example}}

\def\blem{\begin{lem}}
\def\nlem{\end{lem}}

\def\bcor{\begin{cor}}
\def\ncor{\end{cor}}

\def\bexe{\begin{exe}}
\def\nexe{\end{exe}}

\def\bprf{\begin{proof}}
\def\nprf{\end{proof}}

\def\bdes{\begin{description}}
\def\ndes{\end{description}}

\def\benu{\begin{enumerate}}
\def\nenu{\end{enumerate}}

\usepackage{tikz}

\usepackage{geometry}
\begin{document}

 \title[Brownian particules with electrostatic repulsion]
 {Quasi-stationarity of the Dyson Brownian motion with collisions}

\author[A. Guillin]{\textbf{\quad {Arnaud} Guillin$^{\dag}$    }}
\address{{\bf Arnaud Guillin}. Universit\'e Clermont Auvergne, CNRS, LMBP, F-63000 CLERMONT-FERRAND, FRANCE}
 \email{arnaud.guillin@uca.fr}

\author[B. Nectoux]{\textbf{\quad Boris Nectoux$^{\dag}$  }}
\address{{\bf Boris Nectoux}.Universit\'e Clermont Auvergne, CNRS, LMBP, F-63000 CLERMONT-FERRAND, FRANCE}
 \email{boris.nectoux@uca.fr}

\author[L. Wu]{\textbf{\quad Liming Wu$^{\dag}$ \, \, }}
\address{{\bf Liming Wu}. Universit\'e Clermont Auvergne, CNRS, LMBP, F-63000  CLERMONT-FERRAND, FRANCE, and,
Institute for Advanced Study in Mathematics, Harbin Institute of Technology, Harbin 150001, China}
\email{Li-Ming.Wu@uca.fr}

\begin{abstract}
In this work,  we investigate the ergodic behavior  of  a  system of particules, subject to collisions, before it exits a fixed subdomain of its state space. This system is composed of several  one-dimensional   ordered Brownian particules in interaction with   electrostatic repulsions, which is usually referred as the (generalized)  Dyson Brownian motion. The starting points of our analysis are the work [E. Cépa and D.  Lépingle, 1997 Probab. Theory Relat. Fields] which provides existence and uniqueness of such a system subject to collisions via the theory of multivalued SDEs and  a Krein–Rutman   type theorem   derived in   [A. Guillin, B. Nectoux, L. Wu, 2020 J. Eur. Math. Soc.].

\end{abstract}
\maketitle

\vskip 20pt\noindent {\it AMS 2020 Subject classifications.}       37A30, 37A60, 60B10,  60J60, 60J65.

\vskip 20pt\noindent {\it Key words and Phrases.}   Quasi-stationarity, killed multivalued SDEs,   Dyson Brownian motion.  

\section{Introduction}

    \subsection{The model}
   Let $N\ge 1$ and consider the open connected set $\mathscr O:= \{x=(x^1,\ldots, x^N)\in \mathscr R\times \ldots \times\mathscr R, x^1<\ldots<x^N\}$.  We will simply denote $\mathscr R\times \ldots \times\mathscr R $ by $\mathscr R^N$. Note  that $\mathscr O$  is a nonempty unbounded open  convex subset of $\mathscr R^N$.  For $x=(x^1,\ldots, x^N) \in \mathscr R^N$, we consider the confining potential 
   $$V_c(x)=  \sum_{k=1} \mathbf v(x^k),$$
 where $\mathbf v:\mathscr  R\to \mathscr [0,+\infty)$. We assume throughout  this work that $\mathbf v$ is a smooth convex function  such that its derivative $\mathbf v' $ is globally Lipschitz.  We will  also need,  to construct a suitable Lyapunov function,  the following extra assumption on $\mathbf v$. 
For every $\delta>0$ such that 
 \begin{equation}\label{eq.v'}
\lim_{|u|\to +\infty} \mathbf v''(u)/2 -  \delta  |\mathbf v'(u)|^2   =-\infty.
 \end{equation}
 The prototypical exemple of such a function $\mathbf v$ is the quadratic potential $u\in \mathscr R \mapsto  a_{\mathbf v} |u|^2$. 
 Note that $V_c$ is smooth, convex,  and its gradient is globally Lipschitz as well. 
Let us also consider the following interaction potential defined by,  for $\gamma>0$, 
   $$V_I(x) = \left\{
    \begin{array}{ll}
       - \gamma \sum_{1\le i<j\le N} \ln (x^j-x^i)& \mbox{ if  } x\in \mathscr O \\
        + \infty & \mbox{ if } x\notin \mathscr O.
    \end{array}
\right.
$$ 
This proper lower semi-continuous convex function satisfies 
 $${\rm dom}\, (V_I):=\{x\in \mathscr R^N, V_I(x)<+\infty\}=\mathscr O.$$
 Its subdifferential $\partial V_I$ is a simple-valued maximal monotone operator, with 
 $${\rm dom}\, (\nabla V_I)={\rm dom}\, (V_I)=\mathscr O.$$
  Let  $(\mathbf \Omega, \mathcal F, (\mathcal F_t)_{t\ge 0}, \mathbb P)$ be a  filtered probability space, where the filtration satisfies the usual condition (namely it is complete  and right-continuous), and let $(B_t,t\ge 0)$ be standard $\mathscr R^N$ Brownian motion.
 From the theory of multivalued SDEs established in \cite{cepathesis} (see also~\cite{cepa1998problame,zhang2007skorohod,pardouxB}), for all $x\in \overline {\mathscr O}$, there exists a unique strong continuous solution $((X_t, K_t),t\ge 0)$ of 
\begin{equation}\label{eq.L}
dX_t= -\nabla V_c(X_t)dt-dK_t+ dB_t,
\end{equation}
such that:
   \begin{enumerate}
   \item  The process $(K_t,t\ge 0)$ has  a finite variation and $K_0=0$. 
   \item   The process $(X_t,t\ge 0)$ lies in $\overline{\mathscr O}$ for all $t\ge 0$.
   \item For every continuous process $(\alpha,\beta)$ such that for all $s\ge 0$, $(\alpha_s,\beta_s)\in {\rm Gr}\, (\partial V_I)$ (the graph of $\partial V_I$), the measure $\langle X_s-\alpha_s, dK_s-\beta_sds\rangle$ is a.s. non negative on $\mathscr R_+$.
   \end{enumerate}
   We denote by $((X_t(x), K_t(x)),t\ge 0)$ this unique solution and we write $X_t(x)=(x^1_t(x),\ldots, x^N_t(x))$. 
Note that $t\mapsto dK_t(x)$ may \textit{a priori}  not be necessarily absolutely continuous with respect to the    Lebesgue measure on $\mathscr R$. 
One of the main contributions of~\cite{cepaL97} is to  prove that it is actually the case  and that there is no boundary term. More precisely the following result is proved there and it is the starting point of our work.
\begin{thm}[\cite{cepaL97}]\label{th.Cepa}
    For for all $x\in \overline {\mathscr O}$, the following assertions hold true:
    \begin{enumerate}
   \item[{\rm \textbf{i}}.]  $\mathbb P_x[ \{s\ge 0, X_s\in \partial \mathscr O\} \text{ has zero Lebesgue measure}]=1$. 
   \item[{\rm \textbf{ii}}.]  For all $t\ge 0$ and $1\le i<j\le N$, a.s. 
\begin{equation}\label{eq.Cd}
\int_0^t\frac{ds}{x^j_s(x)-x^i_s(x)}<+\infty.
\end{equation} 
   \item[{\rm \textbf{iii}}.]  A.s. $dK_t(x)=\nabla V_I(X_t(x))dt$. 
   \end{enumerate}
   \end{thm}
   
 Note that   Item \textbf{iii} indeed shows that there is   no boundary term  in the process $(K_t,t\ge 0)$. 
  When $\gamma\in (0,1/2)$,   collisions occur a.s.  and never occur when $\gamma \ge 1/2$ (see Lemma~\ref{le.col}).    Item {\rm \textbf{i}} thus  implies that time collisions are very rare in the sense of Lebesgue measure. In particular, since the trajectories of the process are continuous, the set of collision times $\{s\ge 0, X_s\in \partial \mathscr O\}$ is a.s. never dense in any subset of  $\mathscr R_+$ of non zero Lebesgue measure. 
   Item  {\rm \textbf{ii}} in Theorem \ref{th.Cepa} shows that $t\ge 0\mapsto \nabla V_I(X_t(x))$ is locally integrable. If a collision occurs in finite time, this  is thus done in an \textit{integrable  way}, i.e. in a way that preserves the integrability conditions \eqref{eq.Cd}. 
   As initially observed in~\cite{dyson1962brownian}, the process \eqref{eq.L} appears in a natural way in the study of the eigenvalues of a randomly-diffusing symmetric matrix, see~\cite{rogers1993interacting,song2022recent} and references therein.


   \subsection{Purpose of this work and motivation}
   
   The purpose of this work is to study, when collisions occur a.s. (which, we recall,  is the case if and  only if   $\gamma\in (0,1/2)$, see Lemma~\ref{le.col}),   the  long time behavior  of  the process  \eqref{eq.L}    when conditioned not to exit an open subregion $\mathscr U$ of $\overline{\mathscr O}$.  This behavior is strongly linked with the existence and uniqueness of the so-called quasi-stationary distribution of the process \eqref{eq.L} inside $\mathscr U$, see Definition \ref{de.QSD} below. 
    
%
%
   
     The main result of this work is Theorem \ref{th.1} below which describes the long time  behavior of the killed (outside $\mathscr U$) process $(X_t,t\ge 0)$, see also the intermediate results Theorems \ref{th.De}, \ref{th.SF}, \ref{th.Sf}, and \ref{th.TI} which provide information on the regularity of both the killed and the non-killed processes.   We  emphasize that we have no regularity assumption on the boundary of $\mathscr U$, which can be bounded or not.

   To prove Theorem \ref{th.1}, we rely on 
   \cite[Theorem 2.2]{guillinqsd} and  more precisely, we check that all the assumptions of this theorem are valid.   The long time behavior of the killed  process when  the collisions never occur, i.e. when $\gamma>1/2$, can be treated as in \cite[Section 4.2]{guillinqsd2} (see also~\cite{guillinqsd3}) since in such a setting,  the process $(X_t,t\ge 0)$ lies a.s. in $\mathscr O$\footnote{The energy of the system is, when $\gamma \ge 1/2$,  a Lyapunov function which prevents from collisions. }, see indeed Section \ref{sec.seE} below. This is not the case we will focus on here. As already mentioned above, we will rather consider in this work the case when collisions occur a.s. combined with the situation where (which is the   case of interest here):
   $$\mathscr U\cap \partial \mathscr O\neq \emptyset.$$
In particular, to use  \cite[Theorem 2.2]{guillinqsd}, we will have to study the regularity properties (such as the strong Feller property and the topological irreducibility)  of the non-killed and the killed (outside $\mathscr U$) semigroups, see respectively \eqref{eq.NonK} and \eqref{eq.Ki},    when the process starts at a point $x\in  \partial {\mathscr O}\cap \mathscr U$, i.e.  when initially, at least two particules share  the same position (namely starting with a collision) -  see more precisely Theorems \ref{th.De}, \ref{th.SF}, \ref{th.Sf} and \ref{th.TI}. Compared to the framework \cite{guillinqsd2} where collisions never happen, the main difficulty of the analysis here lies in the fact that the drift $\nabla V_c$, though integrable in time (see Item \textbf{ii} above), is infinite on $\partial \mathscr O$. This  prevents from using (at least directly) standard techniques for solutions of stochastic differential equations such as for instance the elliptic regularity theory, the Malliavin calculus,  the Stroock-Varadhan support theorem, or Gaussian upper bounds. 
Moreover, compared to our previous works, we cannot rely on all the tools we developed in~\cite{guillinqsd,guillinqsd2,guillinqsd3}. We will thus need  a little finesse in some places and argue differently. 


 
   \subsection{Notation}
The set $\mathcal B(\overline{\mathscr O} )$ is the Borel $\sigma$-algebra of $\overline{\mathscr O}$, and $b\mathcal B(\overline{\mathscr O})$ is the  space of all bounded and Borel measurable functions $f:\overline{\mathscr O} \to  \mathscr R$ equipped with the sup-norm
 $$\Vert f \Vert_\infty=\sup_{ x\in \overline{\mathscr O} }\vert f( x)\vert.$$
 The set  $\mathcal C_b(\overline{\mathscr O})$ denotes the space of bounded continuous real-valued functions over $\overline{\mathscr O}$.
 Given an initial distribution $\nu$ on $\overline{\mathscr O}$, we write $\mathbb P_\nu(\cdot)=\int_{\overline{\mathscr O}} \mathbb P_{x}(\cdot) \nu(d{x}) $.   The indicator function of a measurable set $\mathcal A$ is  denoted by  $\mathbf 1_{\mathcal A}$. For $T>0$, the space $\mathcal C([0,T],\overline{\mathscr O})$ is the space of continuous functions $g:[0,T]\to \overline{\mathscr O}$, which is equipped with the supremum norm.   For $p\ge 1$ and $k\ge 1$, $L^p(\mathscr R^N,dz)$ stands for the space of functions $g:\mathscr R^N\to \mathscr R^k$ such that $\Vert g\Vert_{L^p}=\int_{\mathscr R^N} |g|^p(z)dz$ is finite (note that we do not refer to the index $k$ in this notation). The set of probability measures over a subset $\mathscr U$ of $\overline{\mathscr O}$ is denoted by $\mathcal P( {\mathscr U})$.  The infinitesimal generator of the process $(X_t,t\ge 0)$ is denoted by $\mathscr L$, i.e.
 $$\mathscr L=\Delta/2 -\nabla V_c\cdot \nabla- \nabla V_I\cdot \nabla.$$
 We end this section by recalling  the notion of quasi-stationary distribution~\cite{collet2011quasi,meleard2012quasi,champagnat2021lyapunov} which is the central object to analyse the long time behavior of  conditioned processes. Such an object is at the heart of the study of biological processes~\cite{collet2011quasi,meleard2012quasi,velleret2019adaptation}  or in the study of metastable dynamics~\cite{faraday,DLLN,lelievre2022eyring}.

\begin{defn}\label{de.QSD}  A measure $ \mu \in \mathcal P({\mathscr U})$ is a  quasi-stationary distribution for the  process $(X_t,t\ge 0)$ (see \eqref{eq.L}) inside $\mathscr U\subset \overline{\mathscr O}$  if $
\mathbb P_{\mu}[X_t\in \mathcal A|t<\sigma_{\mathscr U}]=\mu(\mathcal A)$, $\forall t\ge 0$ and $\forall \mathcal A\in \mathcal B({\mathscr U})$.
\end{defn}

   \subsection{Related results}
 The process \eqref{eq.L} as well as the asymptotic behavior of its empirical measure in the limit $N\to +\infty$ have been investigated in~\cite{rogers1993interacting} in the absence of collision (i.e. when $\gamma>1/2$) and, in the collision case, in~\cite{cepaL97,cepaL2001,cepaL2007} using the theory of multivalued stochastic differential equations~\cite{cepathesis,cepa1998problame} (see also~\cite{zhang2007skorohod,pardouxB}). We also refer to~\cite{demni2009radial} for the study of  the radial Dunkl process and its first hitting time to the boundary Weyl chamber.

 The law of large numbers and the propagation of chaos for its empirical measures have been derived in \cite{li2020law,guillin2022systems}. The ergodic behavior of \eqref{eq.L} has been studied in~\cite{cepaJa,zhangM}, see also~\cite{ren2010large,ren2012regularity} for large deviations principles  in the small noise regime  and the regularity of the invariant measures for solutions to multivalued SDEs. 
  
  The process \eqref{eq.L} is elliptic in the sense that the Brownian noise acts in every direction of $\mathscr R^N$. 
Existence and uniqueness of a quasi-stationary distribution for elliptic diffusions over a  bounded subdomain $\mathscr D$  of $ \mathscr R ^d$ having  sufficiently smooth coefficients over $\overline{\mathscr D}$, is now well-known, see e.g.~\cite{pinsky1985convergence,gong1988killed,champagnat2017general,champagnat2018criteria,le2012mathematical} and references therein. The quasi-stationarity of elliptic and hypoelliptic  processes in the singular potential case and without collision has been investigated in~\cite{guillinqsd2,guillinqsd3}, and in the non singular case in~\cite{guillinqsd,ramilarxiv1,benaim2021degenerate,champagnat2024quasi}.  We also mention~\cite{tough} for existence and uniqueness of the quasi-stationary distribution  
 for the stochastic Fisher-Kolmogorov-Petrovsky-Piscunov on the circle. 
Finally, more materials on quasi-stationary distributions can be found in~\cite{collet2011quasi,meleard2012quasi}.

 \section{Preliminary results}

  \subsection{Collision time}

Let us first notice that  by uniqueness of the strong solution, by standard considerations, $(X_t,t\ge 0)$ satisfies the Markov property, see e.g.~\cite{ cepaJa} or~\cite[Section 1 in Chapter IX]{revuz2013continuous}. We denote its semigroup by 
\begin{equation}\label{eq.NonK}
\text{$P_tf(x)=\mathbb E_x[f(X_t)]$, for $f\in b\mathcal B(\overline{\mathscr O})$, $x\in \overline{\mathscr O}$}, 
\end{equation}
  which is usually called  the non-killed semigroup. The following lemma shows among other things that $P_t$ has the Feller property for all $t\ge 0$. In particular, in view of the proof of~\cite[Theorem 6.17]{le2016brownian},   $(X_t,t\ge 0)$ satisfies the strong Markov property.

 \begin{lem}\label{le.1a}
 Let $x,y\in \overline {\mathscr O}$ and $T>0$. Then,
 $$\mathbb E\Big[ \sup_{t\in [0,T]}|X_t(x)-X_t(y)|^2\Big]\le  |x-y|^2.$$
 \end{lem} 
 \begin{proof}
Using~\cite[Proposition 4.1]{cepa1998problame} (see also~\cite{cepathesis}) together with  the convexity of $V_c$, it holds:  
\begin{align*}
d |X_t(x)-X_t(y)|^2&= -2(X_t(x)-X_t(y))\cdot (\nabla V_c(X_t(x))-\nabla  V_c(X_t(y)) \, dt\\
&\quad - 2(X_t(x)-X_t(y))\cdot (dK_t(x)-dK_t(y)) \\
&\le 0,
\end{align*} 
which proves the desired result. 
 \end{proof}
%

   Let us now introduce the first (positive) collision time:
 $$\sigma_{{\rm col}}:=\inf\{t > 0, x^j_t=x^i_t \text{ for some } i<j\}.$$
    Note that $\sigma_{{\rm col}}$   is also the first positive time the process $(X_t,t\ge 0)$ hits the boundary $\partial \mathscr O$ of $\mathscr O$ (or equivalently, exits $\mathscr O$).  Depending on the parameter $\gamma>0$, collisions between the $N$ particules either always occur or never occur.
   
 \begin{lem}\label{le.col} 
If  $\gamma \ge 1/2$,  then $\mathbb P_{x_0}[ \sigma_{{\rm col}}=+\infty]=1$ for all $ x_0\in  {\mathscr O}$. If $\gamma \in (0, 1/2)$, $\mathbb P_{x_0}[ \sigma_{{\rm col}}<+\infty]=1$ for all $ x_0\in \overline {\mathscr O}$.  
 \end{lem} 
    
   
   \begin{proof}
The case when $\gamma\ge 1/2$ has already been treated in~\cite{rogers1993interacting} using the energy $x\in \mathscr O \mapsto H(x):=V_c(x)+V_I(x)$ as a Lyapunov function which prevents from collisions.
 Assume now that $\gamma\in (0,1/2)$.  Let us prove that $\mathbb P_x[ \sigma_{{\rm col}}<+\infty]=1$. Such a result have been proved in a very similar setting in~\cite{cepaL2001} using Legendre process in the non confining case. We propose here another proof based on a standard argument involving Bessel processes. 
 To this end introduce the first positive collision time between the $\ell$-th particule and the $(\ell+1)$-th particule ($\ell\in \{1,\ldots,N-1\}$):
 $$\sigma_{\ell, \ell+1}:=\inf\{t > 0, x^\ell_t=x^{\ell+1}_t\}, \  \ell \in \{1,\ldots,N-1\}.$$
 Clearly, we have a.s. that  $$\sigma_{{\rm col}}\le \sigma_{\ell, \ell+1}.$$
  In the following,       we omit to write the dependency of the involved  processes in the initial conditions $x_0\in \overline{\mathscr O}$. Set $\wp _t:= |x^{\ell+1}_t-x^\ell_t|^2$ the squared distance between the $\ell$-th particule and the $(\ell+1)$-th particule, $\ell\in \{1,\ldots,N-1\}$. Note that  by Item 
 {\rm \textbf{iii}} in Theorem \ref{th.Cepa}, for all $t>0$, a.s. $\int_0^t|\nabla V_I(X_s(x))|ds<+\infty$, and we can thus use Itô formula, cf. e.g.~\cite[Theorem 4.3.10]{lamberton2007introduction}, to deduce that:
 \begin{align*}
 \wp _t&=  \wp _0+ 2t- 2 \int_0^t(x^{\ell+1}_s-x^\ell_s)(\partial_{x^{\ell+1}}V_c(X_s) - \partial_{x^\ell}V_c(X_s))  ds \\
 &\quad + 2 \int_0^t\sqrt{\wp _s}\, d(B^{\ell+1}_s-B^{\ell}_s)\\
 &\quad +2\gamma  \int_0^t \Big[ \, \sum_{j=1,j\neq \ell}^N\frac{x^{\ell+1}_s-x^\ell_s}{x^{\ell+1}_s-x^j_s} + \sum_{k=1,k\neq \ell+1}^N\frac{x^{\ell}_s-x^{\ell+1}_s}{x^{\ell}_s-x^k_s} \,  \Big]\, ds.
  \end{align*}
  Set for $x=(x^1,\ldots, x^N)\in  {\mathscr O}$:
 \begin{align*}
 \mathbf   b(x)&=-2 (x^{\ell+1}_t-x^\ell_t))(\partial_{x^{\ell+1}}V_c(x) - \partial_{x^\ell}V_c(x))) \\
 &\quad +2\gamma \Big[ \, \sum_{j=1,j\neq \ell}^N\frac{x^{\ell+1}-x^\ell}{x^{\ell+1}-x^j} + \sum_{k=1,k\neq \ell+1}^N\frac{x^{\ell}-x^{\ell+1}}{x^{\ell}-x^k} \,  \Big].
  \end{align*}
  Since $\mathbf v$ is convex,  we have for all $ x=(x^1,\ldots, x^N)\in  {\mathscr O}$:
  \begin{align*}
\mathbf   b(x)& \le  4\gamma  + 2 \gamma \sum_{j\notin\{\ell,\ell+1\}}\Big[\, \frac{x^{\ell+1}-x^\ell}{x^{\ell+1}-x^j} + \frac{x^{\ell}-x^{\ell+1}}{x^{\ell}-x^j}  \, \Big]  \\
& \le  4\gamma  + 2 \gamma \sum_{j<\ell \text{ or } j>\ell+1} (x^{\ell+1}-x^\ell)\,  \frac{x^{\ell+1}-x^\ell}{(x^{\ell+1}-x^j)(x^j-x^{\ell})}   \\
 &\le   4\gamma  - 2 \gamma  \sum_{j<\ell \text{ or } j>\ell+1}     \frac{ |x^{\ell+1}-x^\ell|^2 }{(x^{\ell+1}-x^j)(x^{\ell}-x^j)}   \\
 &\le 4\gamma,
  \end{align*} 
 where, to deduce the last inequality, we have used that  the particules are ordered. 
 By Item \textbf i in Theorem \ref{th.Cepa}, we have a.s. for all $t\ge 0$, $\int_0^t \mathbf b(X_s)ds= \int_0^t\mathbf 1_{X_s\in \mathscr O} \mathbf b(X_s)ds$. 
  Hence, a.s. we have for all $t\ge 0$, 
  \begin{align*}
  \wp _t  &\le    \wp _0+ 2(2\gamma+1)t   + 2 \int_0^t \sqrt{2\wp _s}\, dw_s.    
 \end{align*} 
  In the above  inequality, $w_t:=(B^{\ell+1}_t-B^\ell_t )/\sqrt 2$ is standard real Brownian motion. By  the comparison theorem of Ikeda and Watanabe for  one-dimensional ltô processes~\cite{ikeda1977comparison}, 
 it holds a.s.  
 \begin{equation}\label{eq.B<}
0\le  \wp _t\le \mathscr B_t, \text{  for all $t\ge 0$}, 
 \end{equation}
  where $(\mathscr B_t,t\ge 0)$ solves the equation 
 $$d\mathscr B_t=2(2\gamma+1)dt   + 2   \sqrt{2 \mathscr B_t}\, dw_t, \ y_0= \wp _0= |x^{\ell+1}_0-x_0^\ell |^2\ge 0.$$ 
 The process 
  $(\mathscr B_{t/2},t\ge 0)$ is thus a squared Bessel process of dimension $2\gamma+1\in (1,2)$, see e.g.~\cite[Section 1 in Chapter XI]{revuz2013continuous}. It is well known that 
  the Lebesgue measure of the set $\{t \ge 0, \mathscr B_{t/2}=0\}$ is zero and 
   $(\mathscr B_{t/2},t\ge 0)$   is  reflected infinitely often at the point $0$, see~\cite[Section 11]{revuz2013continuous}. Consequently, this implies that $\mathbb P[ \exists t>0, \mathscr B_t=0]=1$. In conclusion $\mathbb P_{x_0}[\sigma_{ \ell,\ell+1}<+\infty]=1$.   This concludes the proof of the lemma.  
   \end{proof}
   

   Let us mention that it is proved in~\cite{cepaL2007} that  multiple collisions can not occur at any positive time.   In all the rest of this work, we will assume that  $\gamma\in (0,1/2)$ and thus we will work in the case where a.s. collisions occur (see Lemma \ref{le.col}).

    \subsection{On the non-killed semigroup}
    
    In this section, we provide results on the non-killed semigroup we will need to prove Theorem \ref{th.1} below. 
    We start with the following theorem. 
    
     \begin{thm}\label{th.De}
For all $t>0$ and $x\in \overline{\mathscr O}$, $X_t(x)$ has a density w.r.t. the Lebesgue measure $dz$ over $\mathscr R^N$. 
    \end{thm} 
    
    Classical arguments usually employed to prove such a theorem, such as e.g.  those based on the Malliavin calculus or those which rely on the elliptic regularity theory, are difficult to apply directly on the process $(X_t,t\ge 0)$ since $\nabla V_I$ is not Lipschitz over $\overline{\mathscr O}$. 
     Note that Theorem \ref{th.De}  implies that collision times are random except possibly at time $0$ (indeed, $\mathbb P_x[ \sigma_{{\rm col}}=t]\le \mathbb  P_x[ X_t\in \partial \mathscr O]=0$, $t>0$).


    \begin{proof}
    The proof is divided into several steps. 
          \medskip

    \noindent
    \textbf{Step 1.} Preliminary analysis.  Let us recall some results we will need from~\cite{cepathesis,cepaJa}. Set for $n\ge 1$ and $y\in \mathscr R^N$, $\mathbf c_n(y):= -\psi_n(y)-\nabla V_c(y)$, where $\psi_n$ is the smooth convex and  globally Lipschitz vector field defined in~\cite[Eq. (2.20)]{cepaJa}. Denote by  $(X^n_t(x),t\ge 0)$ the solution on $\mathscr R^N$ to 
    $$dX^n_t(x)= \mathbf c_n(X^n_t)dt + dB_t, \ X_0=x.$$
    It is proved in~\cite[Section 2.4]{cepaJa} that for all $T>0$  and $x\in \overline{\mathscr O}$,  as $n\to +\infty$,  $(X^n_t(x),t\in [0,T])$ converges in distribution to $(X_t(x),t\in [0,T])$ in $\mathcal C([0,T], \mathscr R^N)$.  Following the computations led in the proof of~\cite[Proposition 5.5]{cepa2006equations},  for all $T>0$ and all compact subset $K$ of $\overline{\mathscr O}$, there exists $C>0$ such that 
\begin{equation}\label{eq.EXn0}
\forall n\ge 1, x\in K, t\in [0,T], \  \int_0^t\mathbb E_{x}[|\mathbf c_n(X^n_s)|]ds\le C.
\end{equation}

    \noindent
    \textbf{Step 1.}  For $x\in \mathscr R^N$ and $t>0$, let us denote by $f_n^{x} (t,z)$ the density of  
    $X^n_t(x)$, namely $\mathbb P_x[X^n_t(x)\in A]= \int_Af_n^{x}(t,z)dz$, $A\in \mathcal B(\mathscr R^N)$ (note by parabolic elliptic regularity or by Malliavin calculs~\cite{nualart2006malliavin}, $f_n^{x} (t,z)$ indeed exists and is a smooth function of $(t,x,y)\in \mathscr R_+^*\times \mathscr R^N\times \mathscr R^N$). In what follows  $K$ is a fixed compact subset of $\overline{\mathscr O}$ and $T>0$. Note that \eqref{eq.EXn0} rewrites 
    \begin{equation}\label{eq.EXn}
\sup_{n\ge 1, \, t\in [0,T], \, x\in K } \int_0^t  \int_{\mathscr R^N} |\mathbf c_n(z)|f_n^x(s,z)dz \, ds <+\infty.
\end{equation}
Since $\mathbf c_n$ is smooth, for each $n\ge 1$ and $x\in \mathscr R^N$,    the function $(t,y)\in \mathscr R_+^*\times \mathscr R^N\mapsto f_n^{x}(t,y)$ is a smooth solution  of the following parabolic equation over $\mathscr R^N$:
$$\partial_t f_n^{x}= \frac 12 \Delta f_n^{x} -{\rm div} (\mathbf c_n \, f_n^{x}).$$
Let us introduce  the heat kernel $\mathscr  G$ defined by $\mathscr  G(t,y)=t^{-d/2} h(t^{-1/2}y)$, $t>0$, $y\in \mathscr R^N$, where $h(w)=(2\pi)^{-N/2}\exp( -|w|^2/2)$, $w\in \mathscr R^N$. 
In what follows $\star$ denotes the convolution operator in the space variable.  Direct computations shows that for all $t>0$, 
$$\Vert \mathscr  G(t,\cdot)\Vert_{L^p}= t^{\frac N2 (\frac 1p-1)} \Vert h \Vert_{L^p} \text{ and } \Vert  \nabla \mathscr  G(t,\cdot)\Vert_{L^p}= t^{\frac N{2p} -\frac{N+1}{2}} \Vert \nabla  h \Vert_{L^p}.$$
 Now let $\phi_m\in [0,1]$ be a   family of smooth functions indexed by $m\ge 1$ such that $\phi_m(z)=1$ if $|z|\le m$ and $\phi_m(z)=0$ if $|z|>m+1$, and 
    \begin{equation}\label{eq.phiM}
    \sup_{m\ge 1} \sup_{z\in \mathscr R^N} (|\nabla \phi_m|+|\Delta\phi_m|)(z)<+\infty.
        \end{equation} 
        Set $\mathscr G_2(t,y):=\mathscr G(t/2,y)$. 
By Duhamel's formula and after several integrations by parts, we have for   $0<\epsilon<t$,
\begin{align*}
 \phi_m f_n^{x}(t,z)&= 
  \int_\epsilon^t \big [ f_n^{x}(s,\cdot )(\Delta \phi_m/2 +\mathbf c_n \cdot \nabla \phi_m) (\cdot )\big]\star [\mathscr G_2(t-s,\cdot ) ](z)ds\\
&\quad + \int_\epsilon^t \big [ f_n^{x}(s,\cdot)( \nabla   \phi_m +\mathbf c_n \phi_m  )(\cdot)\big]\star [\nabla \mathscr G_2(t-s,\cdot)](z) ds \\
&\quad + [\phi_m(\cdot) f_n^{x}(\epsilon, \cdot)]\star [\mathscr G_2(t-\epsilon,\cdot)](z), \ z\in \mathscr R^N.
\end{align*}
 Hence, by Young’s convolution inequality, we have for   $   \epsilon<t\le T$ and $p\ge 1$, 
 \begin{align*}
 &\Vert \phi_m (\cdot) f_n^{x}(t,\cdot)\Vert_{L^p}  \\
 &\le   C_p   \Vert  h \Vert_{L^p} \int_\epsilon^t \Vert f_n^{x}(s,\cdot )(\Delta \phi_m +\mathbf c_n \cdot \nabla \phi_m) (\cdot )\Vert_{L^1} \, (t-s)^{\frac N2 (\frac 1p-1)}ds\\
 &\quad + C_p   \Vert  \nabla h \Vert_{L^p} \int_\epsilon^t   \Vert f_n^{x}(s,\cdot)(2 \nabla   \phi_m +\mathbf c_n \phi_m  )(\cdot)\Vert_{L^1}\, (t-s)^{\frac N{2p} -\frac{N+1}{2}} ds \\
 &\quad + C_p\Vert \phi_m(\cdot) f_n^{x}(\epsilon, \cdot)\Vert_{L^1}\Vert   h \Vert_{L^p}(t-\epsilon)^{\frac N2 (\frac 1p-1)},
\end{align*}
where $C_p>0$ depends only on $p$. 
  Recall that $f_n^x\ge 0$ and $\Vert f_n^x(s,\cdot) \Vert_{L^1}=1$ for all $s>0$. Using \eqref{eq.EXn} and \eqref{eq.phiM}, for all $\epsilon,T>0$ with $2\epsilon<T$ and all compact subset $K$ of $\overline{\mathscr O}$, there exists $C>0$ such that for all $n\ge 1$, $t\in [2\epsilon,T]$, $x\in K$, and $m\ge 1$, 
 $$ \int_{\mathscr R^N}| \phi_m (z)|^p\, | f_n^{x}(t,z)|^p dz\le C.$$
Letting $m\to +\infty$ and using Beppo Levi's theorem, we deduce that for such $\epsilon>0$, $T>0$, and $K$, it holds:
   \begin{equation}\label{eq.Ess}
   \sup_{n\ge 1, \, t\in [2\epsilon,T],\,  x\in K} \Vert   f_n^{x}(t,\cdot)\Vert_{L^p}  <+\infty.
        \end{equation}

    \noindent
    \textbf{Step 3.} We conclude the proof of Theorem \ref{th.De}.  Fix  $p>1$,  $x\in \overline{\mathscr O}$,  and $t>0$. Thanks to \eqref{eq.Ess}, we can consider a subsequence $n'=n'(t,x)$ and a function $f^x(t,\cdot)$ such that  $f_{n'}^{x}(t,\cdot)\to f^x(t,\cdot)$ weakly in $L^{p}(\mathscr R^N,dz)$ as $n'\to +\infty$.  Hence,  for all $\phi:\mathscr R^N\to \mathscr R$ continuous and with compact support, it holds  in the limit $n'\to +\infty$:
    \begin{equation}\label{eq.Phi-}
    \int_{\mathscr R^N}\phi (z)f_{n'}^{x}(t,z)\to \int_{\mathscr R^N}\phi (z)f^{x}(t,z)dz.
    \end{equation}  
    Thus, since $X_t^n(x)\to X_t(x)$ in distribution (see the first step above), one has for such functions $\phi$,
    \begin{equation}\label{eq.Phi-0}
    \mathbb E_x[\phi(X_t)]=\int_{\mathscr R^N}\phi (z)f^{x}(t,z)dz.
        \end{equation}  
           Note that the previous considerations imply  that for all $t>0$ and $x\in \mathscr R^N$,  $f^{x}(t,z)\ge 0$ $dz$-a.e., $\int_{\mathscr R^N} f^{x}(t,z)dz=1$ and \eqref{eq.Phi-0} holds for $\phi \in f\in b\mathcal B(\mathscr   R^N)$.  The proof of this claim is standard but we write it for sake of completeness.      For all bounded Borel set $\mathscr A$, there exists $(\phi_n)_n\subset   C_c^0(\mathscr R^N)$, $\phi_n\ge 0$, s.t.  $\phi_n\to \mathbf 1_{\mathscr A}$ in $L^{p'}(\mathscr R^N,dz)$ ($p'=p/(p-1)$). By \eqref{eq.Phi-0}, it holds $0\le \int_{\mathscr R^N} \phi_n(z)f^{x}(t,z)dz   \to  \int_{\mathscr R^N} \mathbf 1_{\mathscr A}(z)f^{x}(t,z)dz$. Hence, $\int_{\mathscr R^N} \mathbf 1_{\mathscr A}(z) f^{x}(t,z)dz \ge 0$. Now, set $\mathscr B=\{f^{x}(t,\cdot)<0\}$ and  $\mathscr A_R:= \mathscr B\cap \mathsf B_{\mathscr R^N}(0,R)$, for $R>0$. Then, $  \int_{\mathscr R^N}  \mathbf 1_{\mathscr A_R}(z) f^{x}(t,z)dz =0$. Thus, $\mathscr A_R$ has Lebesgue measure $0$. Letting $R\to +\infty$, one gets  $f^{x}(t,z)\ge 0$ ($dz$-a.e.). Moreover, since $f_{n'}^{x}(t,\cdot)dz\to f^x(t,\cdot)dz$ vaguely, 
  $\int_{\mathscr R^N} f^{x}(t,z)dz\le 1$ (see e.g.~\cite[Chapter 4]{kallenberg}). Thus,  $f^{x}(t,\cdot)$ is integrable. Finally, we have,  
    $$1= \lim_{m\to +\infty } \mathbb E_x[\phi_m(X_t)]=\lim_{m\to +\infty } \int_{|z|\le m } f^x(t,z)dz+o_m(1) \Rightarrow \int_{\mathscr R^N} f^{x}(t,z)dz=1.$$   
    By~\cite[Proposition 4.4 in Chapter 3]{EK}, the probability measures $f^{x}(t,z) dz$ and $P_t(x,dz)$ are equal.
          The proof of Theorem \ref{th.De} is complete.   Note that actually the whole sequence   $(f_{n'}^{x}(t,\cdot))_{n\ge 1}$ converges in $L^{p}(\mathscr R^N,dz)$ as $n\to +\infty$, by uniqueness of its limit point. 
          \end{proof}

    The following result has  been proved in~\cite{zhangM} using a  coupling method combined with a change of probability measure (we also refer to~\cite{cepaJa} when the compact state space case, using a Bismut type formula as in~\cite{peszat1995strong}). We will give another independent proof of this fact, based on   the analysis led in the proof of Theorem \ref{th.De} above. 
    
 \begin{thm}\label{th.SF}
 Let $t>0$. Then, $P_t$ has the strong Feller property, i.e. $P_tf$ is continuous on $\overline{\mathscr O}$ for any $f\in b\mathcal B(\overline{\mathscr O})$. 
 \end{thm}

 
\begin{remarks} 
One powerful    tool to prove the strong Feller property of  a solution to a SDE is to use a Girsanov formula. 
 Let us mention that it is immediate to see that when $N=2$, there is no hope to have a Girsanov formula relating the law of $(X_t,t\ge 0)$ and the law of a standard Brownian motion over $ \mathscr R^2$, when $X_0\in \partial \mathscr O$. 
  
 %
\end{remarks}


   \begin{proof}
    Let $s>0$ and $x\in K$ where $K$ is a compact subset of  $\overline{\mathscr O}$. Consider $0<\epsilon<s/2$ and $T\ge s$. 
Because $f_n^{x}(s,\cdot)\to f^x(s,\cdot)$ weakly in $L^2(\mathscr R^N,dz)$ as $n\to +\infty$ (see \eqref{eq.Phi-}),  one has that, using also the bound   \eqref{eq.Ess}:
    $$\Vert f^{x}(s,\cdot )\Vert_{L^2}\le \liminf_{n'} \Vert f_n^{x}(s,\cdot )\Vert_{L^2}\le   \sup_{n\ge 1, \, t\in [2\epsilon,T],\,  x\in K} \Vert   f_n^{x}(t,\cdot)\Vert_{L^2} <+\infty.$$
    Hence, one gets that for all $0<\epsilon<T$:
      \begin{equation}\label{eq.Ess2}
C^*:= \sup_{ s\in [2\epsilon,T],\,  x\in K} \Vert f^{x}(s,\cdot )\Vert_{L^2}  <+\infty.
        \end{equation}
        Let $\delta>0$. 
Equation \eqref{eq.Ess2} and Theorem \ref{th.De} imply  that for all ${\mathscr A}\in \mathcal B(\overline{\mathscr O})$ such that $\int_{\mathscr A}dz\le \delta$ and  $x\in K$,  
\begin{align*}
P_s(x,{\mathscr A})=\int_{\mathscr A} f^x(s,y)dy \le  \Big[\int_{\mathscr R^N} |f^x(s,y)|^2dy\Big]^{1/2}  \delta^{1/2} \le C^* \delta^{1/2},
\end{align*}
where $P_s(x,{\mathscr A}):=P_s\mathbf 1_{\mathscr A}(x)$. 
Thus, it follows that the  family of probability measures $(P_s(x,dz))_{x\in  \mathscr R^d}$ is  locally uniformly absolutely continuous with respect to  the Lebesgue measure over $\overline{\mathscr O}$, namely for all compact subset $K$ of $\overline{\mathscr O}$, 
 $$\lim_{\delta\to 0} \sup_{{\mathscr A}\in \mathcal B(\overline{\mathscr O}) ,\,  \int_{\mathscr A}dz\le \delta}\, \sup_{x\in K} P_s(x,{\mathscr A})=0.$$
The proof of the theorem is complete using~\cite[Item (b) in Theorem 2.1]{schilling2012strong} (recall  that $P_t$ has the Feller property). 
  \end{proof}

Note that  $\mathbf v\ge 0$ and hence  $V_c\ge 0$. 

   \begin{prop}\label{pr.L} Assume \eqref{eq.v'}. 
 Set for $\alpha >0$, $W=e^{\alpha V_c}\ge 1$. Then, if $1-\alpha/2>0$, $ \mathscr L W(x)/W(x)\to -\infty$ as $|x|\to +\infty$ ($x\in \overline{\mathscr O}$). 
 \end{prop} 
 \begin{proof}   
We have for $x\in \mathscr O$, \begin{align*}
\frac{ \mathscr L W(x)}{W(x)}&=  \alpha \mathscr L V_c(x) + \alpha^2 |\nabla V_c(x)|^2 /2\\
&=  \alpha \Delta V_c(x)/2  - \alpha(1-\alpha/2)  |V_c(x)|^2   -\alpha  \nabla V_I(x)\cdot \nabla V_c(x).
\end{align*}
For $x\in \mathscr O$, recall that   $- \partial_{x_i}V_I(x)= \gamma \sum_{j=1,j\neq i}^N\,  \frac{1}{x^ i-x^j }$. Then, for $x\in \mathscr O$, it  holds:
\begin{align*}
\frac{ \mathscr L W(x)}{W(x)}& =\alpha \sum_{i=1}^N [ \mathbf v''(x^i)/2 -  (1-\alpha/2) |\mathbf v'(x_i)|^2 ]   + \gamma \alpha \sum_{i=1}^N   \sum_{j=1,j\neq i}^N\,  \frac{\mathbf v'(x^i)}{x^i -x^j  }\\
&= \alpha \sum_{i=1}^N [ \mathbf v''(x_i)/2 -  (1-\alpha/2) |\mathbf v'(x^i)|^2 ] + \gamma\alpha    \sum_{  i<j} \Big [  \frac{\mathbf v'(x^i)}{x^i -x^j }  + \frac{\mathbf v'(x^j)}{x^j -x^i }\Big ]\\
&= \alpha \sum_{i=1}^N [ \mathbf v''(x_i)/2 -  (1-\alpha/2) |\mathbf v'(x^i)|^2 ] + \gamma\alpha  \sum_{  i<j}   \frac{\mathbf v'(x^j)-\mathbf v'(x^i)}{x^j -x^i }. 
\end{align*}
Note that    since $\mathbf v'$ is smooth and Lipschitz,   the function  
$$J: (u_1,u_2)\in \{(a,b)\in \mathscr R^2, a< b\} \mapsto \frac{\mathbf v'(u_2)-\mathbf v'(u_1)}{u_2-u_1 }$$
 is   bounded (say by $C_{\mathbf v}>0$)  and has a  continuous (bounded)  extension
 over   $\{(a,b)\in \mathscr R^2, a\le  b\}$, which is still denoted by $J$. 
Then, for all $x\in \overline{\mathscr O}$, 
\begin{align*}
\frac{ \mathscr L W(x)}{W(x)}&= \alpha \sum_{i=1}^N [ \mathbf v''(x_i)/2 -  (1-\alpha/2) |\mathbf v'(x_i)|^2 ]  + \gamma\alpha  \sum_{  i<j} J(x^j,x^i)  \\
&\le 
\alpha \sum_{i=1}^N [ \mathbf v''(x_i)/2 - (1-\alpha/2)  |\mathbf v'(x_i)|^2 ]  + C_{\mathbf v}\gamma\alpha N^2.
\end{align*} 
Thanks to \eqref{eq.v'}, when $x\in \overline{\mathscr O}$ and  $|x|\to +\infty$,  $ \mathscr  L W(x)/W(x)\to -\infty$. The proof of the proposition is complete. 
  \end{proof}

    \section{Main results on the killed process}

  For all nonempty open subset $\mathscr U$ of $\overline{\mathscr O}$, we set 
  $$\sigma_{\mathscr U}:=\inf\{t\ge 0, X_t\notin \mathscr U\},$$
 which   is the first exit time from $\mathscr U$ for the process $(X_t,t\ge 0)$.
  In all this work 
 \begin{equation}\label{eq.HpU}
 \mathscr U \text{  is an nonempty  open subset of } \overline{ \mathscr O},
 \end{equation}
    i.e. there exists an open subset $\mathscr U_*$ of $\mathscr R^N$ such that $\mathscr U= \mathscr U_*\cap \overline{ \mathscr O}$ and $\mathscr U_*\cap \overline{ \mathscr O}\neq \emptyset$.   
Consider the killed (outside $\mathscr U$) semigroup $(P^{\mathscr U}_t,t\ge 0)$ defined by:
\begin{equation}\label{eq.Ki}
P^{\mathscr U}_tf(x)=\mathbb E_x[f(X_t)\mathbf 1_{t<\sigma_{\mathscr U}}],  f\in b\mathcal B( \mathscr U),\ x\in  \mathscr U \text{ and } t\ge 0.
\end{equation}
  Its associated killed renormalized  semigroup is denoted by 
$$ \nu  Q^{\mathscr U}_t (\mathcal A):= \frac{\nu P^{\mathscr U}_t (\mathcal A)}{\mathscr \nu P^{\mathscr U}_t(\mathscr U)}= {\mathbb P_\nu[   X_t \in \mathcal A| t<\sigma_{\mathscr U}]} ,$$
for  $\mathcal A \in \mathcal B({\mathscr U})$ and  $\nu \in \mathcal M_b(\mathscr U)$.


    \begin{thm}\label{th.Sf}
    Assume \eqref{eq.HpU} and $\gamma\in (0,1/2)$. 
 Let $t>0$. Then, $P^{\mathscr U}_t$ has the strong Feller property. 
 \end{thm}

  In the following   $\mathsf B_{\bar{\mathscr O}}(x,r):=\{y\in \overline{\mathscr O}, |y-x|<r\}$ is the open ball in $\overline{\mathscr O}$  of radius $r>0$ centered at $x \in \overline{\mathscr O}$. Note that $\mathsf B_{\bar{\mathscr O} }(x,r)=\mathsf B_{ \mathscr R^N}(x,r)\cap \overline{\mathscr O}$. 
 In view of the proof of~\cite[Theorem 2.2]{chung2001brownian},
it is enough, to deduce Theorem \ref{th.Sf}, to show the following lemma.

\begin{lem}\label{le.22}
For all compact subset $K$ of $\overline{\mathscr O}$ and $T>0$, it holds: 
$$
\sup_{x\in K}\mathbb E_x[\sup_{t\in [0,T]}|K_t(x)|]<+\infty \text{ and } \lim_{s\to 0^+}\sup_{x\in K} \mathbb P_x[\sigma_{\mathsf B_{\bar{\mathscr O}}(x,r)}\le s]=0.
$$
\end{lem}


 \begin{proof}
 Let $K$ be a compact subset of $\overline{\mathscr O}$.
 \medskip
 
 \noindent
 \textbf{Step 1}. Let us prove that for  all $t\in [0,T]$, 
\begin{equation}\label{eq.BbK}
\sup_{x\in K}\mathbb E_x\big [\sup_{t\in [0,T]}|K_t(x)| \big ]<+\infty.
\end{equation}
Let us first prove that for  all $t\in [0,T]$, 
\begin{equation}\label{eq.Xsup0}
\sup_{x\in K}\mathbb E_x\big [\sup_{t\in [0,T]} |X_t(x)|^2\big ]<+\infty.
\end{equation}
Let $a_0$ be a point in the interior of the domain of the maximal monotone operator $\partial V_I$ , namely $a_0\in \mathscr O$, and  let $\gamma_0 >0$ be such that   $\overline{B}(a_0,\gamma_0 )\subset \mathscr O$. Let $\mu_0:=\sup\{|y|, y\in  \partial V_I(z) \text{ where } z\in  \overline{B}(a_0,\gamma_0 )\}=\sup\{|\nabla V_I(z)|, z\in \overline{B}(a_0,\gamma_0 )\}<+\infty$. 
From~\cite[Proposition 4.4]{cepa1998problame} and its proof (see also~\cite{cepathesis}), for all  $x\in \overline{\mathscr O}$ and  all $t\ge 0$, 
\begin{align}
\nonumber
\int_0^t(X_s(x)-a_0)\cdot dK_s(x)&\ge \gamma_0 {\rm V}_0^t(K(x)) -\mu_0 \int_0^t|X_s(x)-a_0|ds\\
\label{eq.Mma}
&\quad -\gamma_0 \mu_0t,
\end{align}
where ${\rm V}_0^t(K(x))$ is the  total variation  of $t\mapsto K_t(x)$ on $[0,t]$. Let $T>0$ be fixed. 
In the following, we simply denote $\sigma_{\mathsf B_{\bar{\mathscr O}}(a_0,n)}=\inf\{t\ge 0, |X_t-a_0|\ge n\}$ by $\sigma_n$. 
The sequence $(\sigma_n)_n$ increases to $+\infty$.  Since $K$ is compact, there exists $n_K$, for all $n\ge n_K$ and $y\in K$, $|y-a_0|< n$. 
Hence, using Itô formula and \eqref{eq.Mma}, we get for  $0\le s\le t\le T$,  $n\ge n_K$,  and $x\in K$:
\begin{align}
\nonumber
\frac 12 |X_{s\wedge \sigma_n}(x)-a_0|^2&= \frac 12 |x-a_0|^2 -\int_0^{s\wedge \sigma_n} (X_u(x)-a_0)\cdot \nabla V_c(X_u(x))du\\
\nonumber
&\int_0^{s\wedge \sigma_n}(X_u(x)-a_0) \cdot dB_u-\int_0^{s\wedge \sigma_n}(X_u(x)-a_0)\cdot dK_u(x) \\
\nonumber
&\quad +   \frac{{s\wedge \sigma_n}}2 \\
\nonumber
&\le   \frac 12 |x-a_0|^2 +\int_0^{s\wedge \sigma_n}(X_u(x)-a_0) \cdot dB_u\\
\nonumber
&-\gamma_0 {\rm V}_0^{s\wedge \sigma_n}(K(x)) +(\mu_0+ |\nabla V_c(a_0)| ) \int_0^{s\wedge \sigma_n}|X_u(x)-a_0|du\\
\label{eq.Lk}
&\quad +\gamma_0 \mu_0T+ \frac T2,
\end{align}
where we have used the convexity of $V_c$. 
Therefore, since $x\in K$, we have: 
\begin{align*}
\sup_{s\in [0,t]} |X_{s\wedge \sigma_n}(x)-a_0|^2&\le   C \Big[1+ \sup_{s\in [0,t]} \big |\int_0^{s\wedge \sigma_n}(X_u(x)-a_0) \cdot dB_u\big| \\
&+  \int_0^{t\wedge \sigma_n}\sup_{s\in [0,u]} |X_s(x)-a_0|\, du\Big],
\end{align*}
where $C>0$ is a constant which is independent of $x\in K$, $n\ge 1$,  and $(s,t)\in [0,T]^2$, and which, in the following, can change from one occurence to another. Taking expectation and using the Cauchy-Schwarz inequality and the fact that $\sqrt z\le z+1$ for $z\ge 0$,  
\begin{align*}
\mathbb E_x\big [\sup_{s\in [0,t]} |X_{s\wedge \sigma_n}(x)-a_0|^2\big ]&\le   C \Big(1+ \mathbb E_x\big [\sup_{s\in [0,t]} \big |\int_0^{s\wedge \sigma_n}(X_u(x)-a_0) \cdot dB_u\big|^2 \big ]\\
&+  \int_0^{t} \mathbb E_x\big [\sup_{s\in [0,u]} |X_s(x)-a_0|^2 \big]du\Big).
\end{align*}
Using the   Burkholder–Davis–Gundy inequalities for the stochastic integral, we get: 
\begin{align*}
\mathbb E_x\big [\sup_{s\in [0,t]} |X_{s\wedge \sigma_n}(x)-a_0|^2\big ]&\le    C \Big(1+ \mathbb E_x\big [  \int_0^{t\wedge \sigma_n}|X_{u}(x)-a_0|^2 du \big]   \\
&+  \int_0^{t} \mathbb E_x\big [ \sup_{s\in [0,u]} |X_{s\wedge\sigma_n}(x)-a_0|^2 \big]  du\Big)\\
&\le   C  \Big(1+    \int_0^{t}\mathbb E_x\big [ |X_{u\wedge\sigma_n} (x)-a_0|^2  \big ]du  \\
&+  \int_0^{t} \mathbb E_x\big [ \sup_{s\in [0,u]} |X_{s\wedge\sigma_n}(x)-a_0|^2\big]du\Big).
\end{align*}
  By Gronwall's inequality~\cite[Lemma 8.4]{le2016brownian} and  since $t\mapsto \mathbb E_x\big [ \sup_{s\in [0,t]} |X_{s\wedge\sigma_n}(x)-a_0|^2\big]$ is bounded (by $n^2$), we deduce  that 
  $$\mathbb E_x\big [\sup_{t\in [0,T]} |X_{t\wedge \sigma_n}(x)-a_0|^2\big ] \le C ,\  \forall x\in K, n\ge 1.$$
  Then Eq. \eqref{eq.Xsup0} follows letting $n\to +\infty$ and applying   Beppo Levi's theorem.  
  We now prove \eqref{eq.BbK}.   By \eqref{eq.Lk},  for all  $0\le t\le T$,  $n\ge n_K$,  and $x\in K$, $\frac 12 |X_t(x)-a_0|^2 \le   C +\int_0^t(X_u(x)-a_0) \cdot dB_u -\gamma_0 {\rm V}_0^{t }(K(x)) +C \int_0^{t }|X_u(x)-a_0|du$. Hence,  
  $${\rm V}_0^{t }(K(x)) \le    C\Big [1+ \int_0^t(X_u(x)-a_0) \cdot dB_u + \int_0^{t }|X_u(x)-a_0|du\Big ].$$ Taking expectation and using \eqref{eq.Xsup0} (note that $\int_0^t(X_u(x)-a_0) \cdot dB_u$ is a martingale, by \eqref{eq.Xsup0}), we thus have that $\mathbb E_x[\sup_{t\in [0,T]}{\rm V}_0^{t }(K(x))]\le C$ for all $x\in K$. This implies \eqref{eq.BbK} and proves the first inequality in Lemma \ref{le.22}.  
 \medskip
 
 \noindent
 \textbf{Step 2}. We now prove the second inequality in Lemma \ref{le.22}. 
Let $\Theta:\mathscr R^N\to [0,1]$ be a smooth function such that  $\Theta=0$ on $\mathsf B_{\mathscr R^N}(0,\delta/2)$ and $\Theta=1$ on $\mathsf B_{\mathscr R^N}^c(0,\delta)$.   Set  $\Theta_x(z)= \Theta(z-x)$.  Note that  for all $x,z\in \mathscr R^N$, $|\nabla \Theta_x|(z)\le \sup_{\mathscr R^N} | \nabla \Theta|$ and $|\Delta \Theta_x|(z)\le \sup_{\mathscr R^N} | \Delta \Theta|$. 
  Note that for all $x\in K$,   by Itô formula,   $(M_{t}^{\Theta_x}(x),t\ge 0)$   is a martingale, where 
  $$M_{t}^{\Theta_x}(x):=\Theta_x(X_{t }(x))-\Theta_x(x)- \int_0^{t }  \mathscr L  \Theta_x(X_{s}(x)) ds.$$   
Let $K_{\delta} $ be the closed  $\delta$-neighborhood of $K$ ($K_{\delta} $ is a compact subset of $\overline{\mathscr O}$)
Since in addition $\Theta_x(x)=0$, we have using the optional stopping theorem,
 \begin{align*}
& \mathbb E_{x}[\Theta_x(X_{t\wedge\sigma_{B(x,\delta)}})]\\
 &\le   \mathbb E_x\Big[\int_0^{t\wedge \sigma_{B(x,\delta)}} \mathscr L \Theta_x(X_{s}(x)) ds\Big]  \\
 &\le    \mathbb E_x\Big[\int_0^{t\wedge\sigma_{B(x,\delta)}} \big  (\, |K_s(x)\cdot \nabla \Theta_x(X_s)|+|\nabla V_c(X_s)\cdot  \nabla \Theta_x(X_s)|\, \big  )\,   ds\Big] \\
 &\quad + t \sup_{\mathscr R^N}|\Delta \Theta|.
 \end{align*}
When $x\in K$ and $s<\sigma_{B(x,\delta)}$, $X_{s}(x)\in K_{\delta}$. Consequently,   for all $x\in K$, 
we have that 
$$ \sup_{x\in K}\mathbb E_x\Big[\int_0^{t\wedge\sigma_{B(x,\delta)}} |\nabla V_c(X_s)\cdot  \nabla \Theta_x(X_s)|   ds\Big]\le t   \sup_{ K_{\delta} }| \nabla V_c| \, \sup_{\mathscr R^N} | \nabla \Theta|.$$
Using \eqref{eq.BbK}, we then deduce that for all $x\in K$,
$$\mathbb E_{x}\big [\Theta_x(X_{t\wedge\sigma_{B(x,\delta)}})\big ]\le t C_K,$$
where $C_K>0$ is a constant independent of $x\in K$ and $t\ge 0$. 
 Note also that $|X_{ \sigma_{B(x,\delta)}}(x)-x|=\delta$.
  Hence,  for all $x\in K$, $\Theta_x(X_{ \sigma_{B(x,\delta)}}(x))=1$ and
  $$ \mathbb P_x\big [\sigma_{B(x,\delta)}\le t\big ]= \mathbb E_{x}\big [\mathbf 1_{\sigma_{B(x,\delta)}\le t} \Theta_x(X_{ \sigma_{B(x,\delta)}})\big ]\le t C_K.$$
  This ends the proof of the lemma.
 \end{proof}

  \begin{thm}\label{th.TI}
       Assume \eqref{eq.HpU}, $\gamma\in (0,1/2)$, and  that  $\mathscr O\cap \mathscr U$ is connected.  
 Let $t>0$. Then,   for all $t>0$, $P^{\mathscr U}_t$ is topologically irreducible, i.e. for all $t>0$, $x,y \in  {\mathscr U}$, and all $r>0$, 
 $$\mathbb P_{x}\big [X_t\in \mathsf B_{\bar{\mathscr O}}(y ,r), t<\sigma_{\mathscr U}\big ]>0.$$
 \end{thm} 
 
 
Notice that  choosing $\mathscr U =\overline {\mathscr O}$ shows that   the non-killed semigroup  $P_t$ is   topologically irreducible. 
Note also that $V_I$ is not locally Lipschitz over $\overline{\mathscr O}$ (and even worse, it is infinite  on $\partial {\mathscr O}$) which prevents from using,  at least directly,  the  standard arguments  to show the irreducibility of semigroups of  solutions to stochastic differential equations which are usually based on the Stroock-Varadhan support theorem~\cite{SV,ben}. 

  \begin{proof}
  To prove Theorem \ref{th.TI}, we need to investigate the probability for the process not to exit $\mathscr U$ before reaching a fixed deterministic ball. 
   As $r>0$, it is enough to consider the case  when 
$$
x\in \mathscr U\text{ and } y \in \mathscr U\cap \mathscr O.
$$ 
The proof is divided into two steps. 
\medskip

\noindent
   \textbf{Step 1.} The case when 
\begin{equation}\label{eq.xy1}
x,y\in   \mathscr U\cap \mathscr O.
\end{equation}
Since $\mathscr O\cap \mathscr U$ is a connected open subset of $\mathscr R^N$ (and thus it is path connected), we can consider   an open  and connected subset  $\mathscr V$ of $\mathscr R^N$ containing $x$ and $y$, and such that $\overline{\mathscr V}\subset \mathscr O \cap \mathscr U$. 
We recall 
that 
$$H(z)=V_c(z)+V_I(z), \ z\in \overline{\mathscr O}.$$  Using standard techniques (see e.g.~\cite{cattiaux09,guillinqsd3}) and since $H$ is smooth over $\overline{\mathscr V}$ (because $\overline{\mathscr V}\subset \mathscr O \cap \mathscr U$), we have the following Girsanov formula: for all $z\in \mathscr V$, $T\ge 0$, and all $F\in b\mathcal B(\mathcal C([0,T], \mathscr V))$, 
 $$
 \mathbb E_z[F(X_{[0,T]})\mathbf 1_{t<\sigma_{\mathscr V}}]= \mathbb E_z[F(B_{[0,T]})\, m_t^B\, \mathbf 1_{t<\sigma^B_{\mathscr V}}],
 $$
 where $\sigma^B_{\mathscr V}(x):= \inf\{t > 0, B_t\notin \mathscr V \}$
 is the   first exit time   of  the  process $(B_t(x)=x+B_t,t\ge 0)$ from $\mathscr V$, where we recall that $B_t=(B_t^1,\ldots,B_t^N)\in \mathscr R^N$ is a standard Brownian motion,  and  
  $$m_t^B(z)  =  \exp\Big[ -  \int_0^t \nabla H(B_s(z)) \cdot d B_s-\frac 12 \int_0^t  |\nabla H(B_s(z))  |^2 ds   \Big].$$
  Note that  $m_t^B(z)\mathbf 1_{t<\sigma^B_{\mathscr V}}(z)$ is a.s. finite. 
  In particular, we have for all $z\in \mathscr V$, $t\ge 0$, and all $f\in b\mathcal B(\mathscr V)$, 
  \begin{equation}\label{eq.Girsanov}
 \mathbb E_z[f(X_t)\mathbf 1_{t<\sigma_{{\mathscr V}}}]= \mathbb E_x\big [f(B_t)\, m_t^B\, \mathbf 1_{t<\sigma^B_{{\mathscr V}}}\big ].
 \end{equation}   
For  any $r>0$, it is well known that  that for all $x\in \mathscr V$ and $t>0$, 
 $$  \mathbb P_x[B_t\in \mathsf B_{ {\mathscr V}}(y ,r), t<\sigma^B_{{\mathscr V}}]>0.$$ 
Indeed,  this can be shown  using the knowledge of the support of the law of the trajectories of a standard Brownian motion. 
   Then, using \eqref{eq.Girsanov} with $f=\mathbf 1_{B_{\mathscr V}(y,r)}$ ($r>0$), 
 we deduce that  for all $t>0$, $x,y\in  {\mathscr U}\cap \mathscr O$, and all $r>0$,
\begin{equation}\label{eq.st1}
\mathbb P_x\big [X_t\in \mathsf B_{\bar{\mathscr O}}(y ,r), t<\sigma_{\mathscr U}\big ] \ge \mathbb P_x[X_t\in \mathsf B_{ {\mathscr V}}(y ,r), t<\sigma_{{\mathscr V}}]>0.
 \end{equation}
This proves  Theorem \ref{th.TI}  in this case, namely when \eqref{eq.xy1} holds.
\medskip 

\noindent
\textbf{Step 2}. We are left to prove Theorem \ref{th.TI} when   $\mathscr U\cap \partial \mathscr O\neq \emptyset$ and   
\begin{equation}\label{eq.xyO}
x \in   \mathscr U\cap \partial \mathscr O \text{ and } y\in \mathscr O,
\end{equation}
namely when $x$ belongs to the boundary of $\mathscr O$ (i.e. when the process  starts with a collision). 
\medskip

\noindent
\textbf{Step 2a}.
 Let us prove that for all $T_F>0$, there exists  $T=T_x\in [0,T_F]$, 
\begin{equation}\label{eq.O-}
\mathbb P_x[X_T\in \mathscr O, T<\sigma_{\mathscr U}]>0.
\end{equation}  
If it is not the case then there exists $T_F>0$, for all $t\in [0,T_F]$, $\mathbb P_x[X_t\in \partial \mathscr O \text{ or }  \sigma_{\mathscr U}\le t ]=1$, and so:
$$  \mathbb P_x\Big [    \bigcap_{t \in [0,T_F]\cap \mathscr Q}\big \{X_t\in \partial \mathscr O \text{ or }  \sigma_{\mathscr U}\le t \}\Big ]=1,$$
 where $\mathscr Q$ stands for the set of rational numbers. This rewrites
 $$  \mathbb P_x\Big [    \forall t \in [0,T_F]\cap \mathscr Q,  \, X_t\in \partial \mathscr O \text{ or }  \sigma_{\mathscr U}\le t  \Big ]=1,$$
 i.e. there exists $\mathbf \Omega_x\subset \mathbf  \Omega$ with $\mathbb P(\mathbf \Omega_x)=1$ such that for all $\omega\in \mathbf \Omega_x$ and all $t \in [0,T_F]\cap \mathscr Q$, either $X_t(\omega)\in \partial \mathscr O$ or $\{X_u(\omega), u\in [0,t]\}\not\subset \mathscr U$.
 In what follows, $\mathbf \Omega_x$  denotes a set of probability $1$ which can change from one occurence to another.  
When $X_0=x$, we have that  a.s. $\sup_{s\in [0,u]} |X_s-x|\to 0$ as 
$u\to 0^+$. Therefore, since $x\in \mathscr U$, we deduce that   when $X_0=x$,   for all $\omega\in \mathbf \Omega_x $,  there exists  $\epsilon(\omega)\in (0,T_F)$, such that $\{X_u(\omega), u\in [0,\epsilon(w)]\}\subset \mathscr U$. Hence, for all $  \omega\in \mathbf \Omega_x $  and all $t\in [0,\epsilon(\omega)]\cap \mathscr Q$, 
$X_t(\omega)\in \partial \mathscr O$. By continuity of the trajectories of the process $(X_t,t\ge 0)$ and because $\partial \mathscr O$ is closed, we deduce that for all $  \omega\in \mathbf \Omega_x $, it holds $$\{X_u(\omega), u\in [0,\epsilon(w)]\} \subset \partial \mathscr O.$$ This contradicts Item \textbf{i} in Theorem \ref{th.Cepa} above. The proof of \eqref{eq.O-} is thus complete. 
\medskip

\noindent
\textbf{Step 2b}. End of the proof of Theorem \ref{th.TI}. 
Let  $t>0$. Consider  $T\in [0,t/2]$   as in \eqref{eq.O-} and set $t=T+u$, $u>0$. 
By the  Markov property of the process $(X_t(x),t\ge 0)$, we have 
\begin{align*}
&\mathbb P_x\big [X_t\in \mathsf B_{\bar{\mathscr O}}(y ,r), t<\sigma_{\mathscr U}\big ]\\
&=\mathbb E_x\Big[\mathbf 1_{T<\sigma_{\mathscr U} }\mathbb P_{X_T }\big [X_{u}\in \mathsf B_{\bar{\mathscr O}}(y ,r), u<\sigma_{\mathscr U}\big ] \Big]\\
&\ge \mathbb E_x\Big[ \mathbf 1_{T<\sigma_{\mathscr U}, X_T\in \mathscr O }  \,  \mathbb P_{X_T }\big [X_{u}\in \mathsf B_{\bar{\mathscr O}}(y ,r), u<\sigma_{\mathscr U}\big ]  \Big].
\end{align*}
 If the last quantity in the previous inequality vanishes,  then a.s. we have:
 $$ \mathbf 1_{T<\sigma_{\mathscr U}, X_T\in \mathscr O }  \,  \mathbb P_{X_T }\big [X_{u}\in \mathsf B_{\bar{\mathscr O}}(y ,r), u<\sigma_{\mathscr U}\big ]  =0.$$
 Using \eqref{eq.O-}, there exists $\mathbf \Omega_{x,T}\subset \mathbf  \Omega$, with $\mathbb P_x[\mathbf \Omega_{x,T}]>0$ such that  for all $  \omega \in \mathbf \Omega_{x,T}$, $T<\sigma_{\mathscr U}(\omega)$ and $X_T(\omega)\in \mathscr O$, and therefore, it holds:
 $$  \mathbb P_{X_T(\omega) }\big [X_{u}\in \mathsf B_{\bar{\mathscr O}}(y ,r), u<\sigma_{\mathscr U}\big ]  =0.$$
Since $X_T(\omega)\in \mathscr O$, this contradicts \eqref{eq.st1}. Hence,  $\mathbb P_x\big [X_t\in \mathsf B_{\bar{\mathscr O}}(y ,r), t<\sigma_{\mathscr U}\big ]>0$, which is the desired result.  The proof of the theorem is complete.  \end{proof}

Note that the proof of the previous theorem  implies that  if  $\overline{\mathscr O}\setminus  \mathscr U$ contains a nonempty open ball in $\mathscr R^N$, for all $x\in \mathscr U$, $\mathbb P_x[\sigma_{\mathscr U}<+\infty]>0$. 

\section{Main result and extension}

\subsection{Main result in the collision case}

Let $W$ be as in Proposition \ref{pr.L} with $1-\alpha/2>0$. 
Before stating the main result of this work, we define for $q>0$, $b\mathcal B_{W^{q}}(\mathscr U)$   as the set of real valued mesurable functions $f$  over $\mathscr U$ such that $f/W^q$ is bounded. We also define  $\mathcal C_{bW^{q}}(\mathscr U):= \{f\in b\mathcal B_{W^{q}}(\mathscr U), f\text{ is continuous}\}$. The spectral radius of bounded linear operator $P$ over $b\mathcal B_{W^{q}}(\mathscr U)$ is denoted by $\mathsf r_{sp}(P|_{b\mathcal B_{W^{q}}(\mathscr U)})$. 
  Using~\cite[Theorem 2.2]{guillinqsd} (and its note) together with Lemma \ref{le.1a},  Proposition \ref{pr.L}, and Theorems \ref{th.SF}, \ref{th.Sf}, and \ref{th.TI}, we  deduce the following result on the behavior of the process $(X_t,t\ge 0)$ conditioned not to exit  a nonempty  open subset $\mathscr U$ of the Polish space $\overline{\mathscr O}$. 
  
 \begin{thm} \label{th.1}
 Assume \eqref{eq.v'}, $\gamma\in (0,1/2)$,  $\mathscr O\cap \mathscr U$ is connected, and $\overline{\mathscr O}\setminus  \mathscr U$ contains a nonempty open ball in $\mathscr R^N$.   Let $p\in (1,+\infty)$. 
  Then,   there exists a unique quasi-stationary distribution  $\rho^*$ for  $(Q^{\mathscr U}_t,t\ge 0)$ in   $  \mathcal P_{W^{1/p}}(\mathscr U)$ and moreover:
\begin{enumerate} 
  \item[{\rm \textbf A}.] 
   For all $t>0$, $P^{\mathscr U}_t:b\mathcal B_{W^{1/p}}(\mathscr U)\to b\mathcal B_{W^{1/p}}(\mathscr U)$ is compact and there exists 
   $$\lambda>0$$
      such that $\mathsf r_{sp}(P^{\mathscr U}_t|_{b\mathcal B_{W^{1/p}}(\mathscr U)} ) =e^{-\lambda t}, \ \forall t>0$. 
   Furthermore,  $\rho^*    P^{\mathscr U}_t =e^{-\lambda t}\rho^*$, for all $t\ge 0$, and $\rho^* (O)>0$ for all nonempty open subsets $O$ of $ \mathscr U$. In addition,  there is a unique   function $\varphi \in \mathcal C_{bW^{1/p}}( \mathscr U)$   such that  $\rho^*   (\varphi )=1$ and $
      P^{\mathscr U}_t \varphi= e^{-\lambda  t} \varphi  \text{ on } \mathscr U, \forall t\ge0$.  
  Moreover, $\varphi>0$   everywhere on $\mathscr U$.
 \item[{\rm \textbf B}.] There exist   $c_1>0$, and   $c_2\ge 1$, s.t.  for all  $t>0$ and all   $\nu\in \mathcal P_{W^{1/p}  }( \mathscr U)$:
 $$
 \sup_{\mathcal A\in\mathcal B(\mathscr U)}\big |  \nu Q^{\mathscr U}_t(\mathcal A)-\rho^*  (\mathcal A)\big |\le c_2 \, e^{-c_1 t} \frac{\nu(W^{1/p} )}{\nu(\varphi)}.
 $$
  \item[{\rm \textbf C}.] For all $x\in \mathscr U$, $\mathbb P_x[\sigma_{\mathscr U}<+\infty]=1$. 
\end{enumerate} 
 \end{thm}  
 The positive real number $\lambda$ is the so-called the principal eigenvalue of $P^{\mathscr U}_t$ over $b\mathcal B_{W^{1/p}}(\mathscr U)$. It can be easily shown that  both $\lambda$ and $\rho^*$ do not depend on $p\in  (1,+\infty)$.

\subsection{Non collision case: extension of Theorem \ref{th.SF} when $\gamma>1/2$}
\label{sec.seE}

 Assume  for simplicity that  $\mathbf v(u)=u^2/2$. Since  $H$ is lower bounded, we assume that $H\ge 1$. When $\gamma\ge  1/2$, since $\mathscr LH\le CH$ over $\mathscr O$ (see the proof of~\cite[Lemma 1]{rogers1993interacting}),  it holds a.s. for all $x\in \mathscr O$,   $X_t(x)\in \mathscr O$ for all $t\ge 0$.  
When $\gamma\ge  1/2$, the assertions of Lemma~\ref{le.1a} and Theorem~\ref{th.SF}   are   valid on the state space $\mathscr O$, providing $X_0(x)=x\in \mathscr O$. The assertions of Theorems~\ref{th.Sf} and~\ref{th.TI} are also still valid when $\gamma\ge 1/2$  and  when  $\mathscr U$ is a subdomain of $\mathscr O$. 
All these claims can be proved   using e.g. the method developed in~\cite{guillinqsd3}. When $\gamma>1/2$ and setting $U=\exp(aH)$,   it holds following the  computations led in \cite[Lemma 1]{rogers1993interacting}:
$$\frac{\mathscr L U(x)}{U(x)}=  - \alpha |\nabla H(x)|^2 + \frac \alpha  2\Delta H(x) + \frac{\alpha^2}{2}  |\nabla H(x)|^2 \to-\infty   \text{ as $x\to \partial \mathscr O\cup\{\infty\}$}, 
$$  
providing $\alpha>0$ is such that $\gamma (1-\alpha/2)>1/2$. 
Hence, for the process  \eqref{eq.L}, all the assertions of Theorem~\ref{th.1} are still valid for    a subdomain $\mathscr U$ of $\mathscr O$   when $\gamma>1/2$ with the Lyapunov function $U $.  
When $\gamma=1/2$, the construction of a Lyapunov function $U$ such that   ${\mathscr L U}/{U}$ is inf-compact over $\mathscr O$   is left for a futur work.

   \medskip
   
   \noindent
   \textbf{Acknowledgement.}
 A. Guillin is supported by the ANR-23-CE-40003, Conviviality, and has benefited from a government grant managed by the Agence Nationale de la Recherche under the France 2030 investment plan ANR-23-EXMA-0001. B.N. is supported by the grant IA20Nectoux from the Projet I-SITE Clermont CAP 20-25 and by the ANR-19-CE40-0010, Analyse Quantitative de Processus Métastables (QuAMProcs).

{\small
 \bibliography{multivoque}  

\begin{thebibliography}{10}

\bibitem{ben}
G.~Ben~Arous, M.~Gradinaru, and M.~Ledoux.
\newblock H{\"o}lder norms and the support theorem for diffusions.
\newblock In {\em Annales de l'IHP Probabilit{\'e}s et statistiques},
  volume~30, pages 415--436, 1994.

\bibitem{benaim2021degenerate}
M.~Bena{\"\i}m, N.~Champagnat, W.~O{\c{c}}afrain, and D.~Villemonais.
\newblock Degenerate processes killed at the boundary of a domain.
\newblock {\em Preprint arXiv:2103.08534}, 2021.

\bibitem{cattiaux09}
P.~Cattiaux, P.~Collet, A.~Lambert, S.~Mart{\'{\i}}nez, S.~M{\'e}l{\'e}ard, and
  J.~San~Mart{\'{\i}}n.
\newblock Quasi-{S}tationary {D}istributions and {D}iffusion {M}odels in
  {P}opulation {D}ynamics.
\newblock {\em The Annals of Probability}, 37(5):1926--1969, 2009.

\bibitem{cepathesis}
E.~Cepa.
\newblock {\em {\'E}quations diff{\'e}rentielles stochastiques multivoques}.
\newblock PhD thesis, Orl{\'e}ans, 1994.

\bibitem{cepa1998problame}
E.~C{\'e}pa.
\newblock Probl{\`e}me de {S}korohod multivoque.
\newblock {\em The Annals of Probability}, 26(2):500--532, 1998.

\bibitem{cepa2006equations}
E.~C{\'e}pa.
\newblock Equations diff{\'e}rentielles stochastiques multivoques.
\newblock In {\em S{\'e}minaire de Probabilit{\'e}s XXIX}, pages 86--107.
  Springer, 2006.

\bibitem{cepaJa}
E.~C{\'e}pa and S.~Jacquot.
\newblock Ergodicit{\'e} d'in{\'e}galit{\'e}s variationnelles stochastiques.
\newblock {\em Stochastics}, 63(1-2):41--64, 2007.

\bibitem{cepaL97}
E.~C{\'e}pa and D.~L{\'e}pingle.
\newblock Diffusing particles with electrostatic repulsion.
\newblock {\em Probability theory and related fields}, 107(4):429--449, 1997.

\bibitem{cepaL2001}
E.~C{\'e}pa and D.~L{\'e}pingle.
\newblock Brownian particles with electrostatic repulsion on the circle:
  {D}yson's model for unitary random matrices revisited.
\newblock {\em ESAIM: Probability and Statistics}, 5:203--224, 2001.

\bibitem{cepaL2007}
E.~C{\'e}pa and D.~L{\'e}pingle.
\newblock No multiple collisions for mutually repelling {B}rownian particles.
\newblock In {\em S{\'e}minaire de Probabilit{\'e}s XL}, pages 241--246.
  Springer, 2007.

\bibitem{champagnat2018criteria}
N.~Champagnat, K.~A. Coulibaly-Pasquier, and D.~Villemonais.
\newblock Criteria for exponential convergence to quasi-stationary
  distributions and applications to multi-dimensional diffusions.
\newblock In {\em S{\'e}minaire de Probabilit{\'e}s XLIX}, pages 165--182.
  Springer, 2018.

\bibitem{champagnat2024quasi}
N.~Champagnat, T.~Leli{\`e}vre, M.~Ramil, J.~Reygner, and D.~Villemonais.
\newblock Quasi-stationary distribution for kinetic {SDE}s with low regularity
  coefficients.
\newblock {\em Preprint arXiv:2410.01042}, October 2024.

\bibitem{champagnat2021lyapunov}
N.~Champagnat and D.~Villemonais.
\newblock Lyapunov criteria for uniform convergence of conditional
  distributions of absorbed {M}arkov processes.
\newblock {\em Stochastic Processes and their Applications}, 135:51--74, 2021.

\bibitem{champagnat2017general}
N.~Champagnat and D.~Villemonais.
\newblock General criteria for the study of quasi-stationarity.
\newblock {\em Electronic Journal of Probability}, 28:1--84, 2023.

\bibitem{chung2001brownian}
K.L. Chung and Z.~Zhao.
\newblock {\em From Brownian {M}otion to Schr{\"o}dinger’s {E}quation},
  volume 312.
\newblock Springer Science \& Business Media, 2001.

\bibitem{collet2011quasi}
P.~Collet, S.~Mart{\'\i}nez, S.~M{\'e}l{\'e}ard, and J.~San~Mart{\'\i}n.
\newblock Quasi-stationary distributions for structured birth and death
  processes with mutations.
\newblock {\em Probability Theory and Related Fields}, 151(1-2):191--231, 2011.

\bibitem{demni2009radial}
N.~Demni.
\newblock Radial {D}unkl processes: existence, uniqueness and hitting time.
\newblock {\em Comptes Rendus. Math{\'e}matique}, 347(19-20):1125--1128, 2009.

\bibitem{faraday}
G.~Di~Ges{\`u}, T.~Leli{\`e}vre, D.~Le~Peutrec, and B.~Nectoux.
\newblock Jump {M}arkov models and transition state theory: the
  quasi-stationary distribution approach.
\newblock {\em Faraday Discussions}, 195:469--495, 2017.

\bibitem{DLLN}
G.~Di~Ges\`u, T.~Leli\`evre, D.~Le~Peutrec, and B.~Nectoux.
\newblock Sharp asymptotics of the first exit point density.
\newblock {\em Annals of PDE}, 5(2), 2019.

\bibitem{dyson1962brownian}
F.~J. Dyson.
\newblock A {B}rownian-motion model for the eigenvalues of a random matrix.
\newblock {\em Journal of Mathematical Physics}, 3(6):1191--1198, 1962.

\bibitem{EK}
S.~N. Ethier and T.G. Kurtz.
\newblock {\em Markov {P}rocesses: {C}haracterization and {C}onvergence}.
\newblock John Wiley \& Sons, 1986.

\bibitem{tough}
W-T.~L. Fan and O.~Tough.
\newblock Quasi-stationary behavior of the stochastic {FKPP} equation on the
  circle.
\newblock {\em Preprint arXiv:2309.10998}, 2023.

\bibitem{gong1988killed}
G.~Gong, M.~Qian, and Z.~Zhao.
\newblock Killed diffusions and their conditioning.
\newblock {\em Probability Theory and Related Fields}, 80(1):151--167, 1988.

\bibitem{guillin2022systems}
A.~Guillin, P.~Le~Bris, and P.~Monmarch{\'e}.
\newblock On systems of particles in singular repulsive interaction in
  dimension one: log and {R}iesz gas.
\newblock {\em Journal de l'{\'E}cole Polytechnique}, 10:867--916, 2023.

\bibitem{guillinqsd3}
A.~Guillin, D.~Lu, B.~Nectoux, and L.~Wu.
\newblock Generalized {L}angevin and {N}os{\'e}-{H}oover processes absorbed at
  the boundary of a metastable domain.
\newblock {\em Preprint arXiv:2403.17471}, March 2024.

\bibitem{guillinqsd}
A.~Guillin, B.~Nectoux, and L.~Wu.
\newblock Quasi-stationary distribution for strongly {F}eller {M}arkov
  processes by {L}yapunov functions and applications to hypoelliptic
  {H}amiltonian systems.
\newblock {\em Journal of the European Mathematical Society}, 26(8):3047--3090,
  2022.

\bibitem{guillinqsd2}
A.~Guillin, B.~Nectoux, and L.~Wu.
\newblock Quasi-stationary distribution for {H}amiltonian dynamics with
  singular potentials.
\newblock {\em Probability Theory and Related Fields}, 185(3-4):921--959, 2023.

\bibitem{ikeda1977comparison}
N.~Ikeda and S.~Watanabe.
\newblock A comparison theorem for solutions of stochastic differential
  equations and its applications.
\newblock {\em Osaka Journal of Mathematics}, 14:619--633, 1977.

\bibitem{kallenberg}
O.~Kallenberg.
\newblock {\em Foundations of {M}odern {P}robability}, volume~2.
\newblock Springer, 1997.

\bibitem{lamberton2007introduction}
D.~Lamberton and B.~Lapeyre.
\newblock {\em Introduction to {S}tochastic {C}alculus {A}pplied to {F}inance}.
\newblock Chapman and Hall/CRC, 1996.

\bibitem{le2012mathematical}
C.~Le~Bris, T.~Leli{\`e}vre, M.~Luskin, and D.~Perez.
\newblock A mathematical formalization of the parallel replica dynamics.
\newblock {\em Monte Carlo Methods and Applications}, 18(2):119--146, 2012.

\bibitem{le2016brownian}
J-F. Le~Gall.
\newblock {\em Brownian {M}otion, {M}artingales, and {S}tochastic {C}alculus}.
\newblock Springer, 2016.

\bibitem{lelievre2022eyring}
T.~Leli{\`e}vre, D.~Le Le~Peutrec, and B.~Nectoux.
\newblock Eyring-kramers exit rates for the overdamped {L}angevin dynamics: the
  case with saddle points on the boundary.
\newblock {\em Preprint arXiv:2207.09284}, 2022.

\bibitem{ramilarxiv1}
T.~Leli{\`e}vre, M.~Ramil, and J.~Reygner.
\newblock A probabilistic study of the kinetic {F}okker--{P}lanck equation in
  cylindrical domains.
\newblock {\em Journal of Evolution Equations}, 22(2):1--74, 2022.

\bibitem{li2020law}
S.~Li, X-D. Li, and Y-X. Xie.
\newblock On the law of large numbers for the empirical measure process of
  generalized {D}yson {B}rownian motion.
\newblock {\em Journal of Statistical Physics}, 181(4):1277--1305, 2020.

\bibitem{meleard2012quasi}
S.~M{\'e}l{\'e}ard and D.~Villemonais.
\newblock Quasi-stationary distributions and population processes.
\newblock {\em Probability Surveys}, 9:340--410, 2012.

\bibitem{nualart2006malliavin}
D.~Nualart.
\newblock {\em The Malliavin calculus and related topics}, volume 1995.
\newblock Springer, 2006.

\bibitem{pardouxB}
E.~Pardoux and A.~Rascanu.
\newblock {\em Stochastic differential equations}.
\newblock Springer, 2014.

\bibitem{peszat1995strong}
S.~Peszat and J.~Zabczyk.
\newblock Strong {F}ller property and irreducibility for diffusions on
  {H}ilbert spaces.
\newblock {\em The Annals of Probability}, pages 157--172, 1995.

\bibitem{pinsky1985convergence}
R.G. Pinsky.
\newblock On {T}he {C}onvergence of {D}iffusion {P}rocesses {C}onditioned to
  {R}emain in a {B}ounded {R}egion for {L}arge {T}ime to {L}imiting {P}ositive
  {R}ecurrent {D}iffusion {P}rocesses.
\newblock {\em The Annals of Probability}, 13(2):363--378, 1985.

\bibitem{ren2012regularity}
J.~Ren and J.~Wu.
\newblock On regularity of invariant measures of multivalued stochastic
  differential equations.
\newblock {\em Stochastic Processes and their Applications}, 122(1):93--105,
  2012.

\bibitem{zhangM}
J.~Ren, J.~Wu, and X.~Zhang.
\newblock Exponential ergodicity of non-{L}ipschitz multivalued stochastic
  differential equations.
\newblock {\em Bulletin des Sciences Math{\'e}matiques}, 134(4):391--404, 2010.

\bibitem{ren2010large}
J.~Ren, S.~Xu, and X.~Zhang.
\newblock Large deviations for multivalued stochastic differential equations.
\newblock {\em Journal of Theoretical Probability}, 23:1142--1156, 2010.

\bibitem{revuz2013continuous}
D.~Revuz and M.~Yor.
\newblock {\em Continuous martingales and Brownian motion}, volume 293.
\newblock Springer Science \& Business Media, 2013.

\bibitem{rogers1993interacting}
L.C.G. Rogers and Z.~Shi.
\newblock Interacting {B}rownian particles and the {W}igner law.
\newblock {\em Probability Theory and Related Fields}, 95:555--570, 1993.

\bibitem{schilling2012strong}
R.~L. Schilling and J.~Wang.
\newblock Strong {F}eller continuity of {F}eller processes and semigroups.
\newblock {\em Infinite Dimensional Analysis, Quantum Probability and Related
  Topics}, 15(02):1250010, 2012.

\bibitem{song2022recent}
J.~Song, J.~Yao, and W.~Yuan.
\newblock Recent advances on eigenvalues of matrix-valued stochastic processes.
\newblock {\em Journal of Multivariate Analysis}, 188:104847, 2022.

\bibitem{SV}
D.W. Stroock and S.R.S. Varadhan.
\newblock On the support of diffusion processes with applications to the strong
  maximum principle.
\newblock In {\em Proceedings of the Sixth Berkeley Symposium on Mathematical
  Statistics and Probability (Univ. California, Berkeley, Calif., 1970/1971)},
  volume~3, pages 333--359, 1972.

\bibitem{velleret2019adaptation}
Aur{\'e}lien Velleret.
\newblock Adaptation of a population to a changing environment under the light
  of quasi-stationarity.
\newblock {\em Preprint arXiv:1903.10165}, 2019.

\bibitem{zhang2007skorohod}
X.~Zhang.
\newblock Skorohod problem and multivalued stochastic evolution equations in
  {B}anach spaces.
\newblock {\em Bulletin des Sciences Math{\'e}matiques}, 131(2):175--217, 2007.

\end{thebibliography}

\bibliographystyle{plain} 

}

 \end{document}